\documentclass[12pt,a4paper]{article}

\usepackage{mathtools}
\usepackage{amsthm}
\usepackage{amsmath}
\usepackage[a4paper, left=2.5cm, right=2.5cm, top=2cm]{geometry}
\usepackage{setspace}
\usepackage{amssymb}

\DeclarePairedDelimiter{\floor} {\lfloor} {\rfloor}

\DeclareRobustCommand{\stirling}{\genfrac\{\}{0pt}{}}

\theoremstyle{plain}
\newtheorem{theorem}{Theorem}
\newtheorem{corollary}[theorem]{Corollary}
\newtheorem{lemma}[theorem]{Lemma}

\theoremstyle{definition}

\theoremstyle{remark}

\numberwithin{equation}{section}

\title{Unordered Factorizations with $k$ Parts}
\author{Jacob Sprittulla \\ sprittulla@alice-dsl.de}
\date{\today}

\begin{document}
	
\maketitle
	
\begin{abstract}
	We derive new formulas for the number of unordered (distinct) factorizations with $k$ parts of a positive integer $n$ as sums over the partitions of $k$ and an auxiliary function, the number of partitions of the prime exponents of $n$, where the parts have a specific number of colors. As a consequence, some new relations between partitions, Bell numbers and Stirling number of the second kind are derived. 
	
	We also derive a recursive formula for the number of unordered factorizations with $k$ different parts and a simple recursive formula for the number of partitions with $k$ different parts.
\end{abstract}

\section{Introduction and main results}
For integers $n \geq 2$, $k \geq 1$ and $l \geq 1$ we consider the following \textit{factorization counting functions}: 
\begin{itemize}
	\item $f(n)$ denotes the number of factorizations of $n$ with parts $\geq 2$, 
	\item $g(n)$ denotes the number of factorizations of $n$ with distinct parts $\geq 2$, 
	\item $f_k(n)$ denotes the number of factorizations of $n$ with exactly $k$ parts $\geq 2$, 
	\item $g_k(n)$ denotes the number of factorizations of $n$ with exactly $k$ distinct parts $\geq 2$, 
	\item $h_l(n)$ denotes the number of factorizations of $n$ with exactly $l$ different parts $\geq 2$, 
	\item $f_{k,l}(n)$ denotes the number of factorizations of $n$ with exactly $k$ parts $\geq 2$, where exactly $l$ parts are different,
	\item $F_k(n)$ denotes the number of factorizations of $n$ with exactly $k$ parts $\geq 1$, 
	\item $G_k(n)$ denotes the number of factorizations of $n$ with exactly $k$ distinct parts $\geq 1$.
\end{itemize}

We count unordered factorizations, where the order of parts is irrelevant.	As an example, we list all factorizations of $n=36=2^2 \cdot 3^2$.
\begin{align*}
f(36) & = 9 =  \# \{ (36),(18,2),(9,4),(12,3),(6,6),(9,2,2),(6,3,2),(3,3,4),(3,3,2,2) \}    \\
g(36) & = 5 = \# \{(36),(18,2),(9,4),(12,3),(6,3,2) \} \\
f_2(36) & =4 = \# \{(18,2),(9,4),(12,3),(6,6) \} \\
g_2(36) & =3 = \# \{(18,2),(9,4),(12,3) \} \\
h_2(36) & =6 = \# \{(18,2),(9,4),(12,3),(9,2,2)(3,3,4),(3,3,2,2) \} \\
f_{3,2}(36) &= 2 = \# \{(9,2,2),(3,3,4) \}
\text{,}
\end{align*}
where $\#$ denotes the number of elements of a set. 

It is easy to see that all the above functions are \textit{prime independent}, meaning that their value is completely determined by the prime signature of $n$. For example, we have $f(12)=f(75)=4$, since $12=2^2 \cdot 3$ has the same prime signature as $75=3 \cdot 5^2$. For this reason, these functions can be considered as multipartitions of the prime exponents of $n$ (where the order of the exponents is irrelevant) as in Andrews \cite[Chapter 12]{And76} and Cheema and Motzkin \cite{Che68}. We will denote the prime exponents of an integer $n=\prod_{i=1}^{\omega} \pi_i^{e_i}$ by $(e_1, \dots, e_{\omega})$, where $\pi_i$ are primes and $\omega=\omega(n)$ denotes the number of distinct prime factors of $n$. 

Recurrence relations for $f_k(n)$ and $g_k(n)$, involving the divisors of $n$, are given by 
\begin{align}
\label{eq:fkn1}
f_k(n) &=  \tfrac{1}{k}   
\sum_{\substack{ d^i \lvert n \\ d \geq 2  }}
f_{k-i}(n/d^i)\\
\label{eq:gkn1}
g_k(n) &=  \tfrac{1}{k}    
\sum_{\substack{ d^i \lvert n \\ d \geq 2  }}
(-1)^{i+1} g_{k-i}(n/d^i)
\text{,}
\end{align}
with boundary conditions	
\begin{align*}
f_{k}(1)=g_k(1) &= 
\begin{cases}
1, & \text{if $k=0$;}\\
0, & \text{otherwise,}
\end{cases} 	
\end{align*}
see Cheema and Motzkin \cite[Theorem 3.II]{Che68} or Subbarao \cite[Theorem 5.3]{Sub04} for \eqref{eq:fkn1} and Knopfmacher and Mays \cite[Equation 15]{Kno05} for \eqref{eq:gkn1}\footnote{More precisely, Knopfmacher and Mays \cite{Kno05} give a recursion for the number of \textit{ordered} distinct $k$-factorizations from which Equation \eqref{eq:gkn1} follows easily.}. Harris and Subbarao \cite[Equation 4]{Har91} give similar formulas for $f(n)$. 

We will use the following notation for partitions. The set of all partitions of an integer $k$ will be denoted by $\mathcal{P}_k$. For a partition $\alpha \in \mathcal{P}_k$, we denote by $\beta=\beta(\alpha)$ the vector of 
$\beta_i = \# \{ \alpha_j =i \}$, $i=1,\dots,a:=\max(\alpha_j)$. Notice that, by definition, we have $\sum_{i=1}^{a} i \beta_i = k$. In Section \ref{sec:factpart}, we will proof the following explicit formula, which allows to calculate the number of (distinct) $k$-factorizations, where parts equal to $1$ are allowed, as a sum over the partitions of $k$. 
\begin{theorem}
	\label{tm:IntroFGkn}
	Let $n \geq 2$ and $k \geq 1$. Then
	\begin{align}
	\label{eq:Fkn}
	F_k(n) &=  
	\sum_{\alpha \in \mathcal{P}_k}  h(\beta)  
	\prod_{j=1}^\omega \nu_\beta(e_j)  \\
	\label{eq:Gkn}
	G_k(n) &=  
	\sum_{\alpha \in \mathcal{P}_k}  h(\beta) (-1)^{\theta(\beta)} 
	\prod_{j=1}^\omega \nu_\beta(e_j),
	\end{align}
	with 
	\begin{align}
	\label{eq:hbeta}
	h(\beta) &= 
	\left(\prod_{i=1}^a i^{\beta_i} \beta_i! \right)^{-1}  \\
	\label{eq:tbeta}
	\theta(\beta) &= 
	\sum_{i=1}^a (1+i) \beta_i 
	\text{.}
	\end{align}
	Further, $\nu_{\beta}(1)=\beta_1$ and for $m \geq 2$ the following recursion holds 
	\begin{align}
	\label{eq:nubetar}
	\nu_{\beta}(m) &= 
	\frac{1}{m} \left( \gamma(m) + 
	\sum_{k=1}^{m-1} \gamma(k) \nu_{\beta}(m-k)
	\right) 
	\text{,} \quad \text{where}  \quad
	\gamma(m) = \sum_{d \lvert m} d\beta_d
	\text{.}
	\end{align}		 		
\end{theorem}	

Since the function $F_k(n)$ and $G_k(n)$ are closely related to the functions $f(n)$, $f_k(n)$, $g(n)$ and $g_k(n)$, see Lemma \ref{lm:fFgG} below, the above theorem can also be used to calculate values of the latter functions. Properties of the auxiliary function $\nu_{\beta}(m)$, the number of partitions of $m$, where the part $i$ can have $\beta_i$ colors, are summarized in Lemma \ref{lm:munu} below. 

The proof of Theorem \ref{tm:IntroFGkn} exploits the multiplicity of $F_k(n)$ and $G_k(n)$, which is the key to separate the prime exponents of $n$ in the factorization counting functions. By evaluating these formulas at primorials, we get equations relating the Stirling numbers of the second kind and the Bell numbers to sums over partitions (Corollary \ref{co:pabe} below).

Recently, Fedorov \cite[Lemma 2]{Fed18} found a similar formula to  \eqref{eq:Fkn} for $F_k(n)$ as a sum over all \textit{compositions} of $k$. We will demonstrate at the end of Section \ref{sec:factpart} that his formula and \eqref{eq:Fkn} are convertible by proving an equation for the harmonic mean of the product of the partial sums of compositions (Lemma \ref{lm:hmc} below). 

Subbarao \cite[Section 4.2]{Sub04} gave a recursive formula for $f_{k,l}(n)$, but his Equation (4.19) contains a slight error. In Section \ref{sec:recfor}, we will give a corrected version, see Theorem \ref{tm:fkln} below, and deduce the following recursive equation for $h_l(n)$, which seems to be new.
\begin{theorem}
	\label{tm:Introhln}
	Let $n \geq 2$, $k \geq 1$ and $l \geq 1$. Then 
	\begin{align}
	h_{l}(n) \log n &= 
	\sum_{\substack{ d^i \lvert n \\ d \geq 2}} 
	\sum_{j=1}^{l} (-1)^{j+1} 
	\binom{i}{j} h_{l-j}(n/d^i)  \log d
	\text{,} 
	\end{align}	
	with boundary condition	
	\begin{align*}
	h_{l}(n) &= 
	\begin{cases}
	1, & \text{if $n=1$ and $l=0$;}\\
	0, & \text{otherwise.}
	\end{cases} 		
	\text{.}
	\end{align*}
\end{theorem}

From this theorem, we deduce equations for $r_l(n)$, the number of partitions with $k$ different parts (Corollary \ref{co:pkln}).

\section{A formula for unordered $k$-factorizations based on partitions}
\label{sec:factpart}
We need some preparations for the proof of Theorem \ref{tm:IntroFGkn}. With the notation for partitions introduced in the previous section, we can restate the exponential formula of Stanley \cite[Corrolary 5.1.6]{Sta01} as  
\begin{align}
\label{eq:ExpFor}
\exp \left( \sum_{k=1}^{\infty} c_k x^k \right)  = 
\sum_{k=0}^{\infty} x^k
\sum_{\alpha \in P_k} \prod_{i=1}^{a} 
\frac{c_i^{\beta_i}}{\beta_i!}
\text{,}
\end{align}
for real $c_k$ and $\lvert x \rvert <1$.

We will subsequently make use of the Dirichlet generating functions (dgf's), given by
\begin{align}
\label{eq:fkdgf}
1+\sum_{n=2}^\infty \sum_{k=1}^\infty 
f_{k}(n) n^{-s} z^k
&=\prod_{n=2}^\infty 
\left( 1 - z n^{-s}  \right)^{-1}   \\
\label{eq:gkdgf}
1+\sum_{n=2}^\infty \sum_{k=1}^\infty \sum_{l=1}^k
g_{k}(n) n^{-s} z^k  		
&=\prod_{n=2}^\infty 
\left( 1 + z n^{-s}  \right)	\\
\label{eq:fkldgf}
1+\sum_{n=2}^\infty \sum_{k=1}^\infty \sum_{l=1}^k
f_{k,l}(n) n^{-s} z^k t^l 
&=\prod_{n=2}^\infty 
\left( 1 + \frac{z t n^{-s} }{1 - z n^{-s} } \right) \\
\label{eq:hldgf}
1+\sum_{n=2}^\infty \sum_{k=1}^\infty \sum_{l=1}^k
h_{l}(n) n^{-s} t^l  
&=\prod_{n=2}^\infty 
\left( 1 + \frac{t}{n^{s}-1} \right) 
\text{,}
\end{align}
see Hensley \cite[Equation 1.4]{Hen87} and Subbarao \cite[Equation 2.4]{Sub04}. 

If we change the set of admissible parts of the factorization counting functions from the set of integers $n \geq 2$ to any non-empty subset 
$A \subseteq \mathbb{N}$, we have to replace the range of the products
on the right hand sides (rhs's) of the equations \eqref{eq:fkdgf}, \eqref{eq:gkdgf} and \eqref{eq:fkldgf} accordingly. We will utilize this  in the proof of the next theorem.

We will need the dgf's of the functions $F_k(n)$ and $G_k(n)$ for specific values of $k$. Knopfmacher and Mays \cite[Theorem 1]{Kno05} gave a dgf for $f_k(n)$. Their approach can easily be generalized to the case of admissible parts $\geq 1$ and to the case of factorizations with distinct parts. 

\begin{theorem}
	\label{tm:FkA}
	Let $k \geq 1$. Then
	\begin{align}
	\label{eq:Fkdgf}
	\mathcal{F}_{k}(s)  & := 1+
	\sum_{n =2}^{\infty}  
	F_{k}(n) n^{-s}  =
	\sum_{\alpha \in \mathcal{P}_k}  h(\beta)  
	\prod_{i=1}^{a} \zeta(is)^{\beta_i}  \\ 
	\label{eq:Gkdgf}
	\mathcal{G}_{k}(s)  & := 1+
	\sum_{n =2}^{\infty}  
	G_{k}(n) n^{-s}   =
	\sum_{\alpha \in \mathcal{P}_k}   h(\beta) (-1)^{\theta(\beta)} 
	\prod_{i=1}^{a} \zeta(is)^{\beta_i},  
	\end{align}
	with $h(\beta)$ and $\theta(\beta)$ given by \eqref{eq:hbeta} and \eqref{eq:tbeta}.
\end{theorem}
\begin{proof}
	We proof \eqref{eq:Gkdgf}. We use \eqref{eq:gkdgf}, 
	$\log(1+x) = \sum_{k=1}^{\infty} \tfrac{1}{k} (-1)^{k+1} x^k$ 
	(for $\lvert x \rvert <1$) and \eqref{eq:ExpFor} to get the following expression for the (joint) dgf of $G_k(n)$ 
	\begin{align*}
	\prod_{n=1}^\infty 
	\left( 1 + z n^{-s}  \right)  
	&= \exp \left( \sum_{n=1}^{\infty}
	\log(1 + t n^{-s})   \right)  \\ 
	&= \exp \left( \sum_{n=1}^{\infty}  \sum_{k=1}^{\infty}
	\tfrac{1}{k} (-1)^{k+1}  t^k n^{-ks} \right)  \\ 
	&= \exp \left( \sum_{k=1}^{\infty} 
	\tfrac{1}{k} (-1)^{k+1}  t^k \zeta(ks) \right) \\
	&= \sum_{k=0}^{\infty} \sum_{\alpha \in P_k} 
	\prod_{i=1}^{l} \frac{c_i^{\beta_i}}{\beta_i!}
	\text{,}
	\end{align*}
	with $c_k:=\tfrac{1}{k} (-1)^{k+1} \zeta(ks)$. Extracting the $k$-th coefficient we find 
	\begin{align*}
	\mathcal{G}_{k}(s) 
	&= \sum_{\alpha \in P_k} \prod_{i=1}^{l} 
	\frac{c_i^{\beta_i}}{\beta_i!} 	  \\
	&= \sum_{\alpha \in P_k} 
	\prod_{i=1}^a
	\frac{(-1)^{(i+1)\beta_i}}{i^{\beta_i} \beta_i!}
	\zeta(is)^{\beta_i} 
	\text{,}
	\end{align*}
	and the claim follows. \\
	The proof of \eqref{eq:Fkdgf} is similar, by using 
	$\log(1-x) = -\sum_{k=1}^{\infty} \tfrac{1}{k} x^k$.
\end{proof}

If we replace $\zeta(s)$ by 
$\zeta_A(s):=\sum_{n \in A} n^{-s}$ in \eqref{eq:Fkdgf} and \eqref{eq:Gkdgf}, for any non-empty subset $A \subseteq \mathbb{N}$, we get the dgf's of the corresponding factorizations counting functions, where the admissible parts are restricted to $A$.

In the next lemma, we list some properties of the function $\mu_{\beta}(n)$, defined by the dgf 
\begin{align}
\label{eq:mudgf}
1+\sum_{n=1}^{\infty} \mu_{\beta}(n) n^{-s} = 
\prod_{i=1}^a \zeta(is)^{\beta_i}
\text{,} 
\end{align}
for a given $\beta$ of length $a$, that appears in \eqref{eq:Fkdgf} and \eqref{eq:Gkdgf}.

\begin{lemma}
	\label{lm:munu}
	Let $n=\prod_{j=1}^{\omega} \pi_j^{e_j}$ and a vector $\beta$ of length $a$ be given. Then
	\begin{align}
	\label{eq:munu}
	\mu_{\beta}(n) = \prod_{j=1}^{\omega} \nu_{\beta}(e_j)
	\text{,} 
	\end{align}
	where
	\begin{align}
	\label{eq:nubeta}
	\nu_{\beta}(m) := \mu_{\beta}(\pi^m)
	\end{align}
	for some prime $\pi$. The ordinary generating function (ogf) of $\nu_{\beta}(m)$ is given by
	\begin{align}
	\label{eq:nuogf}
	1+\sum_{m=1}^{\infty} \nu_{\beta}(m) x^m = 
	\prod_{i=1}^{a} (1-x^i)^{-\beta_i}
	\text{.} 
	\end{align}
	Further, $\nu_{\beta}(1)=\beta_1$ and for $m \geq 2$ the following recursive relation holds 
	\begin{align}
	\label{eq:nubetar2}
	\nu_{\beta}(m) &= 
	\frac{1}{m} \left( \gamma(m) + 
	\sum_{k=1}^{m-1} \gamma(k) \nu_{\beta}(m-k)
	\right)
	\text{,} \quad \text{where}  \quad
	\gamma(m) = \sum_{d \lvert m} d\beta_d
	\text{.}
	\end{align}		 
	
\end{lemma}

\begin{proof}
	First notice that for $i \geq 1$, the $i$-th divisor function $d_i(n)$, which counts the number of ordered factorizations of $n$ with $i$ parts $\geq 1$, has dgf $\zeta(is)$ and is multiplicative and prime independent. Therefore, for every $\beta$, $\mu_{\beta}(n)$ is also multiplicative, by the structure of its dgf. This shows \eqref{eq:munu}.
	
	Form the Euler product of the Riemann zeta function, we get
	\begin{align*}
	1+\sum_{n=1}^{\infty} \mu_{\beta}(n) n^{-s} = 
	\prod_{i=1}^a \prod_{\pi} (1-\pi^{-is})^{-\beta_i}
	\text{.}
	\end{align*}
	
	Restricting both sides of this equation to a single prime $\pi$, using \eqref{eq:nubeta} and substituting $x=\pi^{-s}$, we get \eqref{eq:nuogf}. 
	
	Sloane and Plouffe \cite[p.\ 20]{Slo95} call the series $(\nu_{\beta}(m))_{m \geq 1}$ the Euler transform of $\beta$ and give the recursion \eqref{eq:nubetar2}.		
\end{proof}

It follows from \eqref{eq:nuogf} that $\nu_\beta(m)$ can be interpreted as the number of partitions of $m$ where the part $1$ can appear in $\beta_1$ colors, the part $2$ can appear in $\beta_2$ colors, and so on. For example, we have
\begin{align*}
\nu_{(0,2,1)}(7) = 3 = \# \{ (2_a,2_a,3),(2_a,2_b,3),(2_b,2_b,3) \}
\text{,}
\end{align*}
where the subscripts are indicating the different colors.

Theorem \ref{tm:IntroFGkn} now follows directly from Theorem \ref{tm:FkA} and Lemma \ref{lm:munu}.

Notice that in the formulas of Theorem \ref{tm:IntroFGkn}, the number of addends, and hence the computational effort, grows with $k$ via 
$p(k) \approx \frac{1}{4 k \sqrt{3}} \exp(\pi \sqrt{2k/3})$, by the Hardy-Ramanujan formula. The formulas can be used to derive general equations for specific values of $k$, as in Cheema and Motzkin \cite[Theorem 7.I]{Che68}.

The next lemma covers the relation between the functions $F_k(n)$, $G_k(n)$ and the functions $f_k(n)$, $g_k(n)$, $f(n)$ and $g(n)$. We denote by $\Omega=\Omega(n)$ the number of prime factors of $n$, counted with multiplicity.

\begin{lemma}
	\label{lm:fFgG}
	Let $n \geq 2$ and $k \geq 1$. Then
	\begin{align}
	\label{eq:fFkn}
	f_k(n)  &= F_{k}(n) - F_{k-1}(n)  \\
	\label{eq:gGkn}
	g_k(n)  &= \sum_{i=1}^k (-1)^{k-i} G_i(n) \\
	\label{eq:FOm}
	f(n)    &= F_{\Omega}(n)   \\
	g(n)    &= \sum_{i=0}^{\floor{(\Omega-1)/2}} G_{\Omega-2i}(n) 
	\text{.}
	\end{align}	
\end{lemma}

\begin{proof}
	By viewing all (distinct) $k$-factorizations, where some parts equal $1$, we can directly conclude that $F_k(n)$ and $G_k(n)$ are related to $f_k(n)$ and $g_k(n)$ by 
	\begin{align}
	F_k(n)  &= \sum_{i=1}^k f_i(n) \\
	\label{eq:Ggkn}
	G_k(n)  &= g_{k}(n) + g_{k-1}(n) 
	\text{,}
	\end{align}	
	for all $n \geq 2$ and $k \geq 1$. Solving these equations for $f_k(n)$ and $g_k(n)$ yields \eqref{eq:fFkn} and \eqref{eq:gGkn}.
	
	The equations for $f(n)$ and $g(n)$ are following from $f(n)=\sum_{k=1}^{\Omega} f_k(n)$ 
	and 
	$g(n)=\sum_{k=1}^{\Omega} g_k(n)$.
\end{proof}

For example, from \eqref{eq:Fkn} and \eqref{eq:FOm} we can conclude
\begin{align}
\label{eq:cor11}
f(n) &=
\sum_{\alpha \in \mathcal{P}_\Omega}  h(\beta) 
\prod_{i=1}^\omega \nu_\beta(e_i) 	
\text{.}
\end{align}	
Notice that $\Omega$ may be replaced by any $m \geq \Omega$ in this equation. 

It is well known that, evaluated at primorials $P_n = 2 \cdot 3 \cdots \pi_n$, where $\pi_n$ denotes the $n$-th prime, the factorization counting functions are related to the Stirling number of the second kind, denoted by $\stirling{n}{k}$, and the Bell numbers $B_n$ via  
\begin{align}
\label{eq:fkPn}
f_k(P_n) &= g_k(P_n) = \stirling{n}{k}  \\
\label{eq:fPn}
f(P_n)   &= g(P_n) = B_n 
\text{.}
\end{align}

Evaluating the formulas of Theorem \ref{tm:IntroFGkn} for the $n$-th primorial 
$P_n$, we can derive some interesting identities between partitions and the Stirling and Bell numbers, which we believe are new. Similar formulas for the Bernouilli and the Euler numbers have been found by Vella \cite[Theorem 11]{Vel08}. Recall that 
$\beta_1 = \# \{ \alpha_j =1 \}$.

\begin{corollary}
	\label{co:pabe}
	Let $n \geq 1$, $1 \leq k \leq n$. Let $h(\beta)$ and $\theta(\beta)$ be as in Theorem  \ref{tm:IntroFGkn}. Then 
	\begin{align}
	\label{eq:cor12}
	\sum_{i=1}^k \stirling{n}{i}	&=
	\sum_{\alpha \in \mathcal{P}_k}  h(\beta) \beta_1^n \\
	\label{eq:cor13}
	B_n &= 
	\sum_{\alpha \in \mathcal{P}_n}  h(\beta) \beta_1^n  \\
	\label{eq:cor14}
	\stirling{n}{k} + \stirling{n}{k-1} &=
	\sum_{\alpha \in \mathcal{P}_k}  
	(-1)^{\theta(\beta)} h(\beta) \beta_1^n  \\
	\label{eq:cor15}
	\binom{n}{2} + 1 &=
	\sum_{\alpha \in \mathcal{P}_n}  
	(-1)^{\theta(\beta)} h(\beta) \beta_1^n   
	\text{.}
	\end{align}
\end{corollary} 

\begin{proof}
	Equation \eqref{eq:cor12} follows from \eqref{eq:Fkn}, evaluated at $P_n$, and \eqref{eq:fkPn}.
	Equation \eqref{eq:cor13} follows from \eqref{eq:cor11} and 
	$B_n= \sum_{i=1}^n \stirling{n}{i}$.
	
	Equation \eqref{eq:cor14} follows from \eqref{eq:Gkn}, evaluated at $P_n$, \eqref{eq:fkPn} and \eqref{eq:Ggkn}.
	Equation \eqref{eq:cor15} follows from \eqref{eq:cor14} with $k=n$.
\end{proof} 

We conclude this section by analyzing the relationship between our formula \eqref{eq:Fkn} for $F_k(n)$ and a similar formula of Fedorov \cite[Equation 4]{Fed18} which involves compositions (or \textit{ordered partitions}) of $k$. For a positive integer $k$ and a vector of positive integers $\beta$ with $\sum i \beta_i =k$, we denote the set of all compositions $\alpha$ of $k$ by $\mathcal{C}_k$ and the (sub-) set of compositions of $k$ with the property $\#\{\alpha_j=i\}=\beta_i$ by $\mathcal{C}_{k,\beta}$. Fedorov found the equation

\begin{align}
\label{eq:fedFkn}
F_k(n) &=  \sum_{\alpha \in \mathcal{C}_k} H(\alpha) \mu_\beta(n) 
\end{align}
with
\begin{align}
\label{eq:Halpha}
H(\alpha) &:=  \left(
\alpha_1 (\alpha_1+\alpha_2) \cdots (\alpha_1+\cdots+\alpha_{l}) 
\right)^{-1} 
\text{.}
\end{align}

Since 
$F_k(n) =  \sum_{\alpha \in \mathcal{C}_k} H(\alpha) \mu_\beta(n) = 
\sum_{\alpha \in \mathcal{P}_k} \mu_\beta(n) 
\sum_{\alpha \in \mathcal{C}_{k,\beta}} H(\alpha)$, 
it is straightforward to ask, whether it is possible to deduce \eqref{eq:Fkn} from \eqref{eq:fedFkn} by aggregating all compositions $\alpha \in \mathcal{C}_{k,\beta}$. The following lemma gives a positive answer to that question.

\begin{lemma}
	\label{lm:hmc}
	Let $k \geq 1$ and $\beta$ a vector of length $a$ with 
	$\sum_{i=1}^a i \beta_i =k$. Then
	\begin{align}
	\label{eq:hmc}
	\sum_{\alpha \in \mathcal{C}_{k,\beta} } 
	H(\alpha) &= h(\beta)	
	\text{.}		
	\end{align}
\end{lemma} 
\begin{proof}
	We give a proof by induction on $k$. If $k=1$, lhs and rhs of \eqref{eq:hmc} both equal $1$. 
	
	Let $k \geq 2$. We denote by $\beta^{(j)}$ the vector $(\beta_i-\delta_{i,j})_{i=1,\dots,a}$
	, where the $j$-th element of $\beta$ is reduced by $1$. By extracting the factor $1/k$ in $H(\alpha)$, using that $\alpha_l=j$ for some $j \leq a$ (with $\beta_j \geq 1$), the induction hypothesis and $\sum_j j \beta_j = k$, we get
	\begin{align*}
	\sum_{\alpha \in \mathcal{C}_{k,\beta} }  	H(\alpha) 
	&= \frac{1}{k} 
	\sum_{\substack{ j=1 \\ \beta_j \geq 1} }^a	
	\sum_{\alpha \in \mathcal{C}_{k-j,\beta^{(j)}} } 	
	\left(
	\alpha_1 (\alpha_1+\alpha_2) \cdots (\alpha_1+\cdots+\alpha_{l-1})  \right)^{-1} \\
	&= \frac{1}{k} 
	\sum_{\substack{ j=1 \\ \beta_j \geq 1} }^a	
	h(\beta^{(j)}) \\
	&= \frac{1}{k} 
	\sum_{\substack{ j=1 \\ \beta_j \geq 1} }^a	
	\left(\prod_{i=1}^a i^{\beta_i - \delta_{i,j}} 
	(\beta_i- \delta_{i,j})! \right)^{-1} \\
	&= \frac{1}{k} \sum_{j=1}^a j \beta_j h(\beta)  \\
	&= h(\beta)
	\text{.}
	\end{align*}
	This completes the proof.
\end{proof}

Notice that the number of addends on the lhs of \eqref{eq:hmc} is $\tfrac{(\sum \beta_i)!}{\prod \beta_i!}$. Therefore Lemma \ref{lm:hmc} can also be stated as follows: \textit{For a given $\beta$, the harmonic mean of the product of the partial sums of all compositions 
	$\alpha \in \mathcal{C}_{k,\beta}$ 
	is given by 
	$(\sum \beta_i)! \prod i^{\beta_i}$.}

Fedorov's approach in \cite{Fed18} does not involve the dgf of $F_k(n)$ nor the exponential formula.\footnote{More precisely, Fedorov's starting point is a recursion for $F_k(n)$ similar to \eqref{eq:fkn1}, which can be proved by combinatorial arguments.} 
Therefore his approach together with Lemmata \ref{lm:munu} and \ref{lm:hmc} constitutes an alternative proof of equation \eqref{eq:Fkn}.

\section{Recursive formulas}
\label{sec:recfor}
We begin this section by deriving recursive formulas for the functions 
$f_{k,l}(n)$ and $h_l(n)$. The proof of Theorem \ref{tm:fkln} uses the logarithmic derivatives of the dgf's. This approach was used by Subbarao \cite[Section 4.2]{Sub04}, but the Equation (4.19) derived therein for $f_{k,l}(n)$ contains a slight error, which we correct here.

\begin{theorem}
	\label{tm:fkln}
	Let $n \geq 2$, $k \geq 1$ and $l \geq 1$. Then 
	\begin{align}
	\label{eq:fkln}
	f_{k,l}(n) \log n &= 
	\sum_{\substack{ d^i \lvert n \\ d \geq 2}} 
	\sum_{j=1}^{l} 
	(-1)^{j+1} \binom{i}{j} f_{k-i,l-j}(n/d^i) \log d  \\
	\label{eq:hln}
	h_{l}(n) \log n &= 
	\sum_{\substack{ d^i \lvert n \\ d \geq 2}} 
	\sum_{j=1}^{l} (-1)^{j+1} 
	\binom{i}{j} h_{l-j}(n/d^i)  \log d
	\text{,} 
	\end{align}	
	with boundary conditions	
	\begin{align*}
	f_{k,l}(n) &= 
	\begin{cases}
	1, & \text{if $n=1$ and $k=l=0$;}\\
	0, & \text{otherwise.}
	\end{cases} 	\\	
	h_{l}(n) &= 
	\begin{cases}
	1, & \text{if $n=1$ and $l=0$;}\\
	0, & \text{otherwise.}
	\end{cases} 		
	\text{.}
	\end{align*}
\end{theorem}

\begin{proof}
	The dgf of $f_{k,l}(n)$ is given by
	\begin{align*}
	\mathcal{K}(s,z,t) &:= 
	\sum_{n=2}^\infty \sum_{k=1}^\infty \sum_{l=1}^k
	f_{k,l}(n) n^{-s} z^k t^l  \\
	&= \prod_{n=2}^\infty 
	\left( 1 + \frac{n^{-s} z t}{1-n^{-s} z } \right) \\
	&= \prod_{n=2}^\infty 
	\frac{1 - n^{-s} z (1-t)}{1-n^{-s} z }  
	\text{.}
	\end{align*}
	We calculate the logarithmic derivative of $\mathcal{K}(s,z,t)$ with respect to $s$. Taking logarithms of the rhs, we get
	\begin{align*}
	\log \mathcal{K}(s,z,t) &= 
	\sum_{n=2}^\infty \log \left(  1 - n^{-s} z (1-t)) \right) -
	\sum_{n=2}^\infty \log \left( 1 - n^{-s} z \right) \\
	&= A(1-t,n,s,z) - A(1,n,s,z) \text{,}
	\end{align*}	
	with 
	\begin{align*}
	A(t,n,s,z) &:= 
	\sum_{n=2}^\infty \log \left(  1 - n^{-s} z t) \right) \\
	&= - \sum_{n=2}^\infty \sum_{m=1}^\infty \tfrac{1}{m} 
	n^{-ms} z^m t^m
	\text{.}
	\end{align*}	
	Taking derivative with respect to $s$, we get
	\begin{align*}
	\tfrac{\partial}{\partial s} A(t,n,s,z) &= 
	\sum_{n=2}^\infty \sum_{m=1}^\infty 
	n^{-ms} z^m t^m \log n 
	\text{,}
	\end{align*}
	so that we get, after expanding $(1-t)^m$
	\begin{align*}
	\tfrac{\partial}{\partial s} \log \mathcal{K}(s,z,t)
	&= -\tfrac{\partial}{\partial s} A(1,n,s,z) 
	+\tfrac{\partial}{\partial s} A(1-t,n,s,z) \\
	&= 	\sum_{n=2}^\infty \sum_{m=1}^\infty  
	n^{-ms} z^m  \log n \left( (-1) + 
	\sum_{i=0}^\infty \binom{m}{i} (-1)^i t^i  \right) \\
	&= 	\sum_{n=2}^\infty \log n 
	\sum_{m=1}^\infty n^{-ms} z^m  
	\sum_{i=1}^\infty \binom{m}{i} (-1)^i t^i
	\text{.}
	\end{align*}
	By the definition of $\mathcal{K}(s,z,t)$, we also have
	\begin{align*}
	\tfrac{\partial}{\partial s} \log \mathcal{K}(s,z,t)
	&= \frac{\tfrac{\partial}{\partial s} 
		\mathcal{K}(s,z,t) } 
	{\mathcal{K}(s,z,t)}  \\
	&= \frac{1}{\mathcal{K}(s,z,t)}
	\sum_{n=2}^\infty \sum_{k=1}^\infty \sum_{l=1}^k
	\log n f_{k,l}(n) n^{-s} z^k t^l  
	\text{.}
	\end{align*}
	The recursion \eqref{eq:fkln} now follows by equating the two expressions of the logarithmic derivative of $\mathcal{K}(s,z,t)$, multiplying with $\mathcal{K}(s,z,t)$ and extracting coefficients.
	
	The recursion \eqref{eq:hln} follows by noticing that 
	$\mathcal{H}(s,t)=\mathcal{K}(s,1,t)$, where
	$\mathcal{H}(s,t)$
	denotes the dgf of $h_l(n)$, see equation \eqref{eq:hldgf} above.
\end{proof}

For a prime $\pi$, we denote by $\kappa_{\pi}(n)$ the maximal power of $\pi$ that divides $n$, i.e.\ the number $m$ with $\pi^m \lvert n$ and $\pi^{m+1} \nmid n$. By a standard argument as in Chamberland et al. \cite[Theorem 2]{Cha13}, Equations \eqref{eq:fkln} and \eqref{eq:hln} can be slightly simplified by replacing the term $\log n$ by $\kappa_{\pi}(n)$ on the lhs's and the term $\log d$ by $\kappa_{\pi}(d)$ on the rhs's of the equations.

Analogous recurrence relations can be derived with the same approach as in the above theorem for $f_k(n)$ and $g_k(n)$ as
\begin{align*}
f_k(n) \kappa_{\pi}(n) &=  
\sum_{\substack{ d^i \lvert n \\ d \geq 2  }}
f_{k-i}(n/d^i) \kappa_{\pi}(d) \\
g_k(n) \kappa_{\pi}(n) &=     
\sum_{\substack{ d^i \lvert n \\ d \geq 2  }}
(-1)^{i+1} g_{k-i}(n/d^i) \kappa_{\pi}(d)
\text{.}
\end{align*}	
These equations are similar to \eqref{eq:fkn1} and \eqref{eq:gkn1}. They can be used to derive recursions for $f(n)$ and $g(n)$ by summing up over $k$. 

Evaluating Equations \eqref{eq:fkln} and \eqref{eq:hln} at prime exponents $\pi^n$, we immediately get the following recursions for $p_{k,l}(n)$, the number of partitions of $n$ with exactly $k$ parts, where exactly $l$ parts are different and $r_l(n)$, the number of partitions of $n$ with exactly $k$ different parts.
\begin{corollary}
	\label{co:pkln}
	Let $n \geq 2$, $k \geq 1$ and $l \geq 1$. Then
	\begin{align}
	n p_{k,l}(n) &=
	\sum_{d=1}^{n} d \sum_{j=1}^{l} (-1)^{j+1} 
	\sum_{i=1}^{\floor{n/d}} \binom{i}{j} p_{k-i,l-j}(n-id)    \\
	\label{eq:rln}
	n r_l(n) &=
	\sum_{d=1}^{n} d   \sum_{j=1}^{l}  (-1)^{j+1} 
	\sum_{i=1}^{\floor{n/d}}   \binom{i}{j} r_{l-j}(n-id)   
	\text{,}
	\end{align}
	with the boundary conditions	
	\begin{align*}
	p_{k,l}(n) &= 
	\begin{cases}
	1, & \text{if $n=k=l=0$;} \\
	0, & \text{otherwise,}
	\end{cases} \\		
	r_{l}(n) &= 
	\begin{cases}
	1, & \text{if $n=l=0$;} \\
	0, & \text{otherwise.}
	\end{cases} 	
	\end{align*}			
\end{corollary}
\begin{proof}
	For any prime $\pi$, we have $p_{k,l}(n)=f_{k,l}(\pi^n)$ and $r_{l}(n)=h_{l}(\pi^n)$.
\end{proof}

Another recursion for $r_l(n)$ can be derived by a standard combinatorial approach without using the generating function. We denote by $r_{l,j}(n)$ the number of partitions of $n$ with exactly $l$ parts, all parts being $\geq j$; it follows that $r_l(n)=r_{l,1}(n)$. 
\begin{theorem}
	Let $n \geq 1$, $l \geq 1$ and $j \geq 1$. Then
	\begin{align}
	\label{eq:rljn}
	r_{l,j}(n)  &= 
	r_{l,j+1}(n) +
	\sum_{i=1}^{\floor{n/j}}  r_{l-1,j+1}(n-ij)  
	\end{align}
	with boundary condition
	\begin{align*}
	r_{l,j}(n) &= 
	\begin{cases}
	0, & \text{if $n<0$ or $l<0$ or $jl>n$;} \\
	0, & \text{if $n=0$ and $l \geq 1$;} \\
	1, & \text{if $n=l=0$.}
	\end{cases} 		
	\end{align*}			
\end{theorem}	
\begin{proof}
	We call a partition of $n$ with $l$ different parts $\geq j$ an $(n,l,j)$-partition. Let $0 \leq i \leq \floor{n/j}$ be the number of parts equal to $j$ in such a partition. The number of $(n,l,j)$-partitions with no part equal to $j$ is $r_{l,j+1}(n)$. For $i \geq 1$, the number of $(n,l,j)$-partitions with exactly $i$ parts equal to $j$ is 
	$r_{l-1,j+1}(n-ij)$. 
	This proves the theorem.
\end{proof}

Other recursive equations for the function $r_l(n)$ are known. For example, Merca \cite[Corollary 1.2]{Mer16} gave a recursion based on $q_k(n)$, the number of distinct $k$-partitions; another recursion, involving the auxiliary function 
$a_{i,j}  = \sum_{d \lvert j} \binom{d-1}{i-1}$,
was also found by Merca \cite[Theorem 1.1]{Mer17}.

Formula \eqref{eq:rln} seems to be the first recurrence relation for $r_l(n)$ which avoids other auxiliary functions. Formula \eqref{eq:rljn} seems to be computationally more efficient than the other variants mentioned.

\bibliographystyle{plain}

\begin{thebibliography}{99}
	
	\bibitem{And76}	
	G. Andrews, 
	\textit{The Theory of Partitions, Encyclopedia of Mathematics and its Applications 2} 
	(1976), Addison Wesley.
	
	\bibitem{Cha13}
	M. Chamberland, C. Johnson, A. Nadeau, and B. Wu, 
	Multiplicative Partitions, 
	\textit{Electron. J. Combin.} 
	\textbf{20} (2013), \#P57.	
	
	\bibitem{Che68}
	M. S. Cheema and T. S. Motzkin, 
	Multipartitions and multipermutations, 
	\textit{Combinatorics (Proc. Sympos. Pure Math., Vol. XIX, Univ. California, Los Angeles, Calif., 1968)}, 
	American Mathematical Society, Rhode Island, (1971), 39--70.
	
	\bibitem{Fed18}
	G. V. Fedorov, 
	On Unordered Multiplicative Partitions, 
	\textit{Dokl. Math.} 
	\textbf{98} (2018), 607--611.
	
	\bibitem{Har91}	
	V. C. Harris, M. V. Subbarao, 
	On Product Partitions of Integers, 
	\textit{Canad. Math. Bull.}  
	\textbf{34} (1991), 474--479.
	
	\bibitem{Hen87}
	D. Hensley, 
	The distribution of the number of factors in a factorization, 
	\textit{J. Number Theory} 
	\textbf{26} (1987), 179--191.
	
	\bibitem{Kno05}	
	A. Knopfmacher and M. E. Mays, 
	A survey of factorisation counting functions, 
	\textit{Int. J. Number Theory} 
	\textbf{1} (2005), 563--581.
	
	\bibitem{Kno06}	
	A. Knopfmacher and M. E. Mays, 
	Dirichlet generating functions and factorization identities, 
	\textit{Congressus Numerantium} 
	\textbf{173} (2005), 117--127.
	
	\bibitem{Mer16}
	M. Merca,
	A note on the partitions involving parts of $k$ different magnitudes,
	\textit{J. Number Theory} 
	\textbf{162} (2016), 23--34.
	
	\bibitem{Mer17}
	M. Merca,
	On the number of partitions into parts of different magnitudes,
	\textit{Discrete Math.} 
	\textbf{340} (2017), 644--648.
	
	\bibitem{Slo95}
	N. J. A. Sloane and S. Plouffe,  
	\textit{The Encyclopedia of Integer Sequences},
	San Diego, CA: Academic Press, 1995.
	
	\bibitem{Sta01}
	R. P. Stanley, 
	\textit{Enumerative combinatorics volume 2}, 
	Cambridge University Press, 2001.
	
	\bibitem{Sub04}
	M. V. Subbarao, 
	Product partitions and recursion formulae, 
	\textit{Int. J. Math. Math. Sci.} 
	\textbf{33} (2004),
	1725--1735.
	
	\bibitem{Vel08}
	D. C. Vella, 
	Explicit formulas for Bernoulli and Euler numbers, 
	\textit{Integers}   
	\textbf{8} (2008), \#A01.
	
\end{thebibliography}

\bigskip
\hrule
\bigskip

\noindent 2010 {\it Mathematics Subject Classification}: Primary 11A51;
Secondary 05A17.

\noindent \emph{Keywords:}
number of unordered factorizations with $k$ parts, multiplicative partition, sum over partition, Bell number, Stirling number of the second kind.

\bigskip
\hrule
\bigskip
		
\end{document}